\documentclass[12pt,a4paper,twoside,reqno]{amsart}

\usepackage{hyperref,url}
\usepackage{amssymb,latexsym}
\usepackage{amsmath,amsthm}
\usepackage[latin1]{inputenc}
\usepackage{enumerate}
\usepackage[all]{xy} \CompileMatrices
\usepackage{amscd}

\title[T-dual Hull-Strominger system]{T-dual solutions of the Hull-Strominger system on 
non-K\"ahler threefolds}
\author[M. Garcia-Fernandez]{Mario Garcia-Fernandez}

\address{Dep. Matem\'aticas, Universidad Aut\'onoma de Madrid, and Instituto de Ciencias Matem\'aticas (CSIC-UAM-UC3M-UCM), Cantoblanco, 28049 Madrid, Spain}
\email{mario.garcia@icmat.es}

\SelectTips{cm}{12} 
\theoremstyle{plain}
\newtheorem{theorem}{Theorem}[section]
\newtheorem{lemma}[theorem]{Lemma}

\theoremstyle{definition}
\newtheorem{definition}[theorem]{Definition}
\newtheorem{definition-theorem}[theorem]{Definition-Theorem}

\newtheorem{remark}[theorem]{Remark}

\numberwithin{equation}{section} 

\DeclareMathOperator{\tr}{tr}

\def\dbar{\bar\partial}

\newcommand{\CC}{{\mathbb C}}

\newcommand{\RR}{{\mathbb R}}
\newcommand{\ZZ}{{\mathbb Z}}

\begin{document}

\begin{abstract}
We construct new examples of solutions of the Hull-Strominger system on non-K\"ahler torus bundles over K3 surfaces, with the property that the connection $\nabla$ on the tangent bundle is Hermite-Yang-Mills. With this ansatz for the connection $\nabla$, we show that the existence of solutions reduces to known results about moduli spaces of slope-stable sheaves on a K3 surface, combined with elementary analytical methods. We apply our construction to find the first examples of T-dual solutions of the Hull-Strominger system on compact non-K\"ahler manifolds with different topology.
\end{abstract}

\maketitle


\section{Introduction}

Let $X$ be a compact complex manifold of dimension $n$ endowed with a holomorphic volume form $\Omega$. Let $\mathbb{V}$ be a smooth hermitian vector bundle over $X$. Let $\alpha$ be a real constant. The Hull-Strominger system, for a hermitian metric $\omega$ on $X$ and a unitary connection $A$ on $\mathbb{V}$, is given by
\begin{equation}\label{eq:HS}
\begin{split}
F_A^{0,2} = 0, \qquad F_A \wedge \omega^{n-1} & = 0,\\
d(\|\Omega\|_\omega \omega^{n-1}) & = 0,\\
dd^c \omega - \frac{\alpha}{4}(\tr R_\nabla \wedge R_\nabla - \tr F_A \wedge F_A) & = 0. 
\end{split}
\end{equation} 
The first line in \eqref{eq:HS} describes the Hermite-Yang-Mills condition for the connection $A$, while the second implies that $\omega$ is conformally balanced. In the last equation, known as the Bianchi identity, there is an ambiguity in the choice of a metric connection $\nabla$ in the tangent bundle of the manifold, back to its 
origins in heterotic string theory \cite{HullTurin,Strom}. Mathematically, \eqref{eq:HS} can be regarded as a family of systems of partial differential equations, often called \emph{Hull-Strominger systems}, parametrized by the different \emph{ans\"atze} for the connection $\nabla$. Unless otherwise stated, in this paper we will consider only the case $n=3$.

The first solutions of the Hull-Strominger system on compact non-K\"ahler manifolds were found in a remarkable work by Fu and Yau \cite{FuYau2,FuYau}, decades after the appearence of these equations in the string theory literature \cite{HullTurin,Strom}. Taking $\nabla$ as the Chern connection of $\omega$ and a suitable ansatz for the metric, Fu and Yau reduced the system to a complex Monge-Amp\`ere type equation for a scalar function on a K3 surface, and solved it by means of hard analytical techniques (see also \cite{Phong4,PPZ2}). Since then, and starting from the seminal work by Li and Yau \cite{LiYau}, the subsequent studies of different analytical and geometrical aspects of the Hull-Strominger system have provided an important impulse to complex non-K\"ahler geometry. 
At the present time, the existence and uniqueness problem 
seems to be far distant, and some of the tools that may lead one day to a general solution have just started being developed \cite{grt2,Phong}. It seems therefore necessary to obtain larger classes of examples where we can test these interesting proposals.

The purpose of this paper is twofold. Our first goal is to construct new solutions of the Hull-Strominger system on the compact non-K\"ahler threefolds originally considered by Fu and Yau. These old geometries are torus fibrations over K3 surfaces, endowed with conformally balanced hermitian metrics constructed by Goldstein and Prokushkin \cite{GoPro}. 
With the natural ansatz that $\nabla$ is Hermite-Yang-Mills \cite{BeRoo,FIVU,Hull2,Ivan09} (or, equivalently, an SU$(3)$-instanton)
\begin{equation}\label{eq:HYMintro}
R_\nabla^{0,2} = 0, \qquad R_\nabla \wedge \omega^{2} = 0,
\end{equation}
we show that the existence of solutions reduces to solve the Laplace equation on a K3 surface $S$, combined with known results about moduli spaces of slope-stable sheaves on $S$. In this way, in Theorem \ref{thm1} we find a new class of simple solutions of the Hull-Strominger which, as we will see, have some remarkable properties.

Our second goal has to do with the qualitative properties of these new solutions, in relation to a fascinating proposal about mirror symmetry for the Hull-Strominger system by Yau \cite{Yau2005}. This new incarnation of mirror symmetry, in close relation with the \emph{(0,2)-mirror symmetry} in the physics literature \cite{MelPle}, should have very different geometric features than the more familiar mirror symmetry which exchanges the A and B models of a Calabi-Yau manifold $X$ \cite{WittenMS}. A novel ingredient is the presence of a holomorphic vector bundle $V$ over $X$ with
$$
ch_2(X) = ch_2(V).
$$
Another interesting feature is that \emph{$(0,2)$-mirrors} can arise in tuples \cite{GGG}, rather than in pairs. According to \cite{ABS}, one way of constructing a \emph{$(0,2)$-mirror} for a pair $(X,V)$ is to implement \emph{T-duality} \cite{Buscher1,RoVer} along a U$(1)$-isometry direction on $X$. Of course, if $X$ is K\"ahler Ricci-flat, compact, and simply connected, there are no continuous isometries, and we are left with the case of tori. For solutions of the Hull-Strominger in a non-K\"ahler complex threefold such U$(1)$-isometries may nevertheless exist. This is indeed the case for our solutions in Theorem \ref{thm1}.

Motivated by this, our second goal is to find the first examples of T-dual solutions of the Hull-Strominger system on compact non-K\"ahler manifolds. 
To achieve our goal, we rely on a general result in previous work by the author \cite[Theorem 7.6]{GF3}, where it was proved 
that the solutions of the Hull-Strominger system 
with the Hermite-Yang-Mills ansatz \eqref{eq:HYMintro} for the connection $\nabla$ are exchanged by \emph{heterotic T-duality}. This notion of T-duality adapted to the Hull-Strominger system was introduced by Baraglia and Hekmati \cite{BarHek}, building on previous work by Bouwknegt, Evslin, and Mathai \cite{BEM}  and Cavalcanti and Gualtieri \cite{CaGu}, and also on the original physical references \cite{BeJO,MaSc,Narain}. In Theorem \ref{th:TdualtityHSGP} we prove that the solutions in Goldstein-Prokushkin threefolds constructed in Theorem \ref{thm1} are exchanged by T-duality, whereby changing the fundamental group of the complex manifold. We speculate that our T-dual solutions correspond to \emph{$(0,2)$-mirrors} (see Section \ref{sec:Tdual}).

The difficulties in finding these first examples, which may explain the time delay between the present paper and \cite{GF3}, have to do with the fact that our methods only apply to solutions with the Hermite-Yang-Mills ansatz \eqref{eq:HYMintro} for the connection $\nabla$. In particular, we have not been able to apply \cite[Theorem 7.6]{GF3} to the original solutions found by Fu and Yau \cite{FuYau}, which take $\nabla$ in \eqref{eq:HS} to be the Chern connection (see Remark \ref{rem:FuYau}). This further motivates our existence result Theorem \ref{thm1}. To the knowledge of the author, the only examples of solutions of the Hull-Strominger system with the Hermite-Yang-Mills ansatz are contained in \cite{AGF1,AGF2,FeiYau,FIVU,OUVi,PPZ1}. The biggest pool of solutions to the present day is provided by the ansatz for $\nabla$ given by the Chern connection
\cite{FHZ,FeiYau,FIVU,FTY,OUVi,Phong2,Phong3,Phong4,PPZ1,PPZ2}, which includes the first solutions found by Fu, Li, Tseng, and Yau. Solutions with the Bismut connection, the Hull connection, and other choices for the connection $\nabla$ were found in \cite{CCDLMZ,Fei0,Fei1,FeiYau,FIUVa,FIVU,OUVi,UVi}.

The approach to the Hull-Strominger system in the mathematics literature, given by finding exact solutions to the PDE \eqref{eq:HS}, is often in tension with the physics approach to this question \cite{MelMiSe}, which considers perturbative solutions given by a formal expansion in the parameter $\alpha$. Somewhat miraculously, the Hermite-Yang-Mills ansatz for the connection $\nabla$ has some special properties which reconcile these two points of view and bring them closer \cite{MaSp}. Mathematically, this particular ansatz introduces a symmetry between the connections $\nabla$ and $A$ which allows us to use Courant algebroid techniques \cite{GF,GF3,grt2} or to reduce \eqref{eq:HS} to a simpler PDE. On the physics side, this symmetry has been used for a long time to successfully deal with perturbative solutions \cite{BeRoo,OssaSvanes}.

T-duality for perturbative solutions of \eqref{eq:HS} on torus bundles over K3 surfaces has been studied in the string theory literature in \cite{EsvlinMina,Israel}. 
These solutions were predicted in \cite{DRS} via a chain of dualities, known as heterotic/F-theory duality, for which solid mathematical underpinnings have not yet been provided (see \cite{FMW,HeLiYau} for some interesting progress in this direction). 
The first rigorous construction of the solutions is due to Fu and Yau in their seminal work \cite{FuYau}. Similarly, our main result Theorem \ref{th:TdualtityHSGP} seems to provide the first rigorous examples of T-dual solutions of the Hull-Strominger system on compact non-K\"ahler manifolds with different topology. It is interesting to observe that the solutions constructed in Theorem \ref{thm1} are formally very similar to the perturbative solutions obtained by Becker and Sethi \cite{BeSe} following \cite{DRS}. We speculate that this resemblance may not be causal, and could possibly help us to sheed new light on the mathematical aspects of the heterotic/F-theory duality principle.

\vspace{12pt}
\noindent
\textbf{Acknowledgments:} I wish to thank X. De la Ossa, T. Fei, P. Hekmati, J. McOrist, R. Minasian, D. Phong, E. Svanes, C. Tipler and L. Ugarte for helpful discussions and comments about the manuscript. Theorem \ref{thm1} answers a question posed to the author by Minasian during the `String Theory, Geometry and String Model Building' at the Mainz Institute for Theoretical Physics (MITP). I thank the organizers and the participants of this activity for the stimulating conversations, and the MITP for its hospitality and support.

\section{New solutions on Goldstein-Prokushkin threefolds}\label{sec1}

Let $S$ be a K3 surface endowed with a K\"ahler Ricci-flat metric $\omega_S$ and a holomorphic volume form $\Omega_S$. We do not assume that $S$ is algebraic. Let $\omega_1$ and $\omega_2$ be anti-self-dual $(1,1)$-forms on $S$ such that 
$$
[\omega_i/2\pi] \in H^2(S,\ZZ).
$$
Let $X$ be the total space of the fibred product of the principal $U(1)$-bundles determined by $[\omega_1/2\pi]$ and $[\omega_2/2\pi]$. Then, $X$ is a principal $T^2$-bundle over $S$
\begin{equation}\label{eq:T2fibration}
p \colon X \to S
\end{equation}
and we can choose a connection $\upsilon$ on $X$ with curvature 
$$
i F_\upsilon = (\omega_1,\omega_2) \in \Omega^2(S,\RR^2).
$$
Identifying $\RR^2 \cong \CC$ we construct a $T^2$-invariant complex one-form $\sigma$ on $X$ such that
\begin{equation}\label{eq:sigma}
d\sigma = \omega_1 + i \omega_2.
\end{equation}
Consider $X$ endowed with the almost complex structure determined by the complex $3$-form
$$
\Omega = p^*\Omega_S \wedge \sigma.
$$
This complex structure is integrable, that is, $d\Omega = 0$, and 
the corresponding complex threefold $X$ is non-K\"ahler unless $\omega_1 = \omega_2 = 0$ (see \cite[Theorem 1]{GoPro}).

Let $u$ be a smooth function on $X$ and $t > 0$ a positive real constant. Then, the hermitian metric 
\begin{equation}\label{eq:GPansatz}
\omega_{t,u} = p^*(e^u \omega_S) + \frac{i t}{2} \sigma \wedge \overline{\sigma},
\end{equation}
satisfies the \emph{conformally balanced equation} \cite{GoPro} (see also \cite[Lemma 12]{FuYau})
$$
d(\|\Omega\|_{\omega_{t,u}} \omega_{t,u} \wedge \omega_{t,u}) = 0.
$$
Note that the constant $t$ in \eqref{eq:GPansatz} can be absorbed in the definition of $u$ by an homothety, recovering the original ansatz in \cite{GoPro}. Nonetheless, the size of the fibres of \eqref{eq:T2fibration} will play an important role in Section \ref{sec:Tdual}, so we have decided to make this dependence explicit.

Consider the smooth complex vector bundle $\mathbb{V} \to S$ given by
$$
\mathbb{V} = T^{1,0}X/T^2.
$$
We stress the fact that we do not consider any preferred holomorphic structure on $\mathbb{V}$. One has $p^* \mathbb{V} = T^{1,0}X$ and (see \cite[Proposition 8]{FuYau})
$$
c_1(\mathbb{V}) = 0, \qquad c_2(\mathbb{V}) = 24 = c_2(S).
$$
Our first basic observation is contained in the following result.

\begin{lemma}\label{lemma:HYMV}
Let us fix an arbitrary hermitian metric on $\mathbb{V}$. Then, the smooth complex vector bundle $\mathbb{V}$ admits a unitary Hermite-Yang-Mills connection with respect to $\omega_S$.
\end{lemma}

\begin{proof}
By the Donaldson-Uhlenbeck-Yau Theorem \cite{Don,UY} we need to show that the moduli space of slope-stable sheaves on the (possibly non-projective) K3 surface $S$ with Mukai vector 
$$
v(\mathbb{V}) = (3,0,-21)
$$ 
is non-empty. A sufficient condition for the moduli space with Mukai vector $v = (v_0,v_1,v_2)$ to be non-empty is that the Mukai pairing satisfies (see \cite{PeregoToma}) 
$$
(v,v) = v_1^2 - 2v_0v_2 \geq 0.
$$
Since $(v(\mathbb{V}),v(\mathbb{V})) = 126$, the statement holds.
\end{proof}

Consider the bilinear form defined by Poincar\'e duality
$$
Q \colon H^2(S,\ZZ) \times H^2(S,\ZZ) \to \ZZ
$$
and denote $Q(\beta,\beta) = Q(\beta)$, for $\beta \in H^2(S,\ZZ)$. Let $\alpha \in \mathbb{R} \backslash \{0\}$ such that
\begin{equation}\label{eq:integral}
\frac{t}{\alpha}(Q([\omega_1/2\pi]) + Q([\omega_2/2\pi])) \in \mathbb{Z}.
\end{equation}

\begin{remark}\label{rem1}
Notice that, for $j= 1,2$, $Q([\omega_j/2\pi]) = - 2k_j$ for some $0 < k_j \in \ZZ$ (see \cite{FuYau}), and hence we can choose e.g. $\alpha = \pm 2t$. 
\end{remark}

Let $\mathbb{W}$ be a smooth complex vector bundle over $S$ with rank $r$, $c_1(\mathbb{W}) = 0$ and second Chern class
\begin{equation}\label{eq:c2Wbis}
c_2(\mathbb{W}) =  24 + \frac{t}{\alpha}(Q([\omega_1/2\pi]) + Q([\omega_2/2\pi])).
\end{equation}
Definition \eqref{eq:c2Wbis} implies the more familiar `anomaly cancellation' condition for the pull-back bundle $p^*\mathbb{W}$ on the cohomology of $X$
$$
c_2(p^*\mathbb{W}) - c_2(X) = 0 \in H^4(X,\RR).
$$
For the case that the $T^2$-bundle $X$ is non-trivial, this follows from the isomorphism $H^2(X,\RR) \cong H^2(S,\RR)/\langle [\omega_1,\omega_2] \rangle$ (see \cite{GoPro}).

\begin{lemma}\label{lemma:HYMW}
Let us fix an arbitrary hermitian metric on $\mathbb{W}$. Then, the smooth complex vector bundle $\mathbb{W}$ admits a unitary Hermite-Yang-Mills connection with respect to $\omega_S$, provided that
\begin{equation}\label{eq:ineqslope}
r \leqslant  c_2(\mathbb{W}).
\end{equation}
\end{lemma}

\begin{proof}
We note that the Mukai vector of $\mathbb{W}$ is given by
$$
v(\mathbb{W}) = (r,0,r-c_2(\mathbb{W})).
$$ 
Arguing as in Lemma \ref{lemma:HYMV}, to prove the statement it suffices that
$$
(v(\mathbb{W}),v(\mathbb{W})) = -2r (r - c_2(\mathbb{W})) \geqslant 0 
$$
which is equivalent to \eqref{eq:ineqslope}.
\end{proof}

\begin{remark}\label{rem2}
If $\alpha/t \in \{2,4\}$ and $Q([\omega_1/2\pi]) = Q([\omega_2/2\pi]) = -2$, then \eqref{eq:ineqslope} is satisfied for any $r \leqslant 22$.
\end{remark}

Assuming that \eqref{eq:ineqslope} is satisfied, by Lemma \ref{lemma:HYMV} and Lemma \ref{lemma:HYMW} we can choose $\nabla$ and $A$ unitary Hermite-Yang-Mills connections with respect to $\omega_S$ on, respectively, $\mathbb{V}$ and $\mathbb{W}$. Consider the pull-back bundles $p^*\mathbb{V} = T^{1,0}X$ and $p^*\mathbb{W}$ over $X$, endowed with the pull-back connections $p^*\nabla$ and $p^*A$. A straightforward calculation shows that $p^*\nabla$ and $p^*A$ are Hermite-Yang-Mills connections with respect $\omega_{t,u}$ in \eqref{eq:GPansatz} (see \cite[Lemma 16]{FuYau}) for any choice of $t$ and $u$. Note that, at this point, the connection $p^*\nabla$ is unitary with respect to a fixed reference $T^2$-invariant hermitian metric on $T^{1,0}X$. We will deal with this issue in the proof of Theorem \ref{thm1}.

With the ansatz $(\omega_{t,u},p^*\nabla,p^*A)$ for our solutions, the Hull-Strominger system \eqref{eq:HS} reduces to the condition
\begin{equation}\label{eq:BI}
2i \partial \dbar (\omega_{t,u}) = \frac{\alpha}{4} p^*(\tr R_\nabla \wedge R_\nabla - \tr F_A \wedge F_A).
\end{equation}
Applying \cite[Lemma 13]{FuYau} we have
$$
2i \partial \dbar (\omega_{t,u}) = \frac{1}{4}p^*\Big{(}(- 2\Delta(e^u) + t\|\omega_1\|^2_{\omega_S} + t\|\omega_1\|_{\omega_S}^2)\omega_S^2\Big{)}
$$
where $\Delta$ denotes the Laplacian of $\omega_S$, with the convention $\Delta f \omega_S^2 = 4i \dbar \partial f \wedge \omega$. Therefore, \eqref{eq:BI} is equivalent to the following Laplace equation on $S$:
\begin{equation}\label{eq:Laplace}
- 2\Delta(e^u) \omega_S^2 = \alpha \tr R_\nabla \wedge R_\nabla - \alpha \tr F_A \wedge F_A - t \|\omega_1\|^2_{\omega_S} \omega_S^2 - t \|\omega_1\|_{\omega_S}^2 \omega_S^2.
\end{equation}
Integrating the right hand-side of the equation and using that
$$
Q([\omega_j/2\pi]) = - \frac{1}{8\pi^2}\int_S \|\omega_1\|^2_{\omega_S} \omega_S^2,
$$
we obtain that the sufficient and necessary condition to solve \eqref{eq:BI} is \eqref{eq:c2Wbis}.


\begin{theorem}\label{thm1}
Let $S$ be a K3 surface endowed with a K\"ahler Ricci-flat metric $\omega_S$. Let $\omega_1$ and $\omega_2$ be anti-self-dual $(1,1)$-forms on $S$ such that $[\omega_i/2\pi] \in H^2(S,\ZZ)$, and consider the corresponding principal $T^2$-bundle $p \colon X \to S$. Let $\alpha,t \in \RR$ with $t > 0$ satisfying \eqref{eq:integral}, and let $\mathbb{W}$ be a smooth complex hermitian vector bundle with rank $r$, $c_1(\mathbb{W}) = 0$, second Chern class \eqref{eq:c2Wbis}, and satisfying 
$$
r \leqslant c_2(\mathbb{W}).
$$
Then there exists a smooth function $u$ on $S$ and Hermite-Yang-Mills unitary connections $p^*\nabla$ on $T^{1,0}X = p^*\mathbb{V}$ and $p^*A$ on $p^*\mathbb{W}$ such that $(\omega_{t,u},p^*\nabla,p^*A)$ is a solution of the Hull-Strominger system \eqref{eq:HS}. Furthermore, we can assume that the connection on the tangent bundle is $\omega_{t,u}$-unitary.
\end{theorem}

\begin{proof}
The proof follows from Lemma \ref{lemma:HYMV}, Lemma \ref{lemma:HYMW}, and the previous construction. For the last part of the statement, note that if $p^*\nabla$ is unitary with respect to a fixed hermitian metric $h_0$ on $T^{1,0}X$, we can choose a complex gauge transformation $g$ taking $h_0$ to the hermitian metric given by $\omega_{t,u}$, and hence $(\omega_{t,u},g \cdot p^*\nabla,A)$ is also a solution which satisfies the desired property.
\end{proof}

\begin{remark}\label{rem:FuYau0}
Crucially, the previous result uses the Hermite-Yang-Mills ansatz for the connection $\nabla$ in order to reduce \eqref{eq:BI} to the simple Laplace equation \eqref{eq:Laplace} (cf. Remark \ref{rem:FuYau}).
\end{remark}

Concrete families of examples which fulfill the hypothesis of Theorem \ref{thm1} can be easily constructed from Remark \ref{rem1} and Remark \ref{rem2}. The previous theorem applies also when $\omega_1 = 0 = \omega_2$, providing in this case non-K\"ahler solutions of the Hull-Strominger system on the product $S \times T^2$. It is interesting to observe that these solutions cannot be obtained from the deformation argument in \cite{AGF1}, since the holonomy of any Calabi-Yau metric reduces to $\operatorname{SU}(2)$.

\section{Examples of T-dual solutions}\label{sec:Tdual}

In this section we study the behaviour of the solutions constructed in Theorem \ref{thm1} under T-duality, providing first examples of T-dual solutions of the Hull-Strominger system on compact non-K\"ahler manifolds.

We start by recalling the notion T-duality relevant for the Hull-Strominger system from \cite{BarHek}. Let $G$ be a compact semisimple Lie group with Lie algebra $\mathfrak{g}$. We fix a symmetric non-degenerate invariant bilinear form
$$
\langle \cdot , \cdot \rangle \in S^2(\mathfrak{g}^*)
$$ 
and consider the corresponding biinvariant Cartan three-form on $G$
$$
\sigma^3 = - \frac{1}{6}\langle \omega , [\omega,\omega] \rangle,
$$
where $\omega$ denotes the $\mathfrak{g}$-valued Maurer-Cartan one-form on the group. 
Let $M$ be a smooth manifold. Let $P$ be a smooth principal $G$-bundle over $M$.

\begin{definition}[\cite{Redden}]\label{def:stringclass}
A \emph{real string class} on $P$ is a class $\tau \in H^3(P,\RR)$ which restricts to $[\sigma^3] \in H^3(G,\RR)$ on the fibres of $P$. 
\end{definition}

Real string classes form a torsor over $H^3(M,\RR)$, that we denote by $H^3_{str}(P,\RR)$, where the action is defined by pullback and addition \cite[Prop. 2.16]{Redden}
$$
\tau \to \tau + \pi^*[H].
$$
Here $\pi \colon P \to M$ is the canonical projection on the principal bundle $P$ and $[H] \in H^3(M,\RR)$. Given a real string class and a connection $\theta$ on $P$, we can always find a representative of the class which is of the form
\begin{equation}\label{eq:hatH}
\hat H = \pi^* H + CS(\theta),
\end{equation}
where $CS(\theta)$ denotes the Chern-Simons three-form
\begin{equation*}\label{eq:CS}
CS(\theta) = -  \frac{1}{6}\langle \theta,[\theta,\theta]\rangle + \langle F_\theta \wedge \theta\rangle \in \Omega^3(P).
\end{equation*}
Consequently, real string classes are $G$-invariant classes on $P$. As it can be inferred from $dCS(\theta) = \langle F_\theta \wedge F_\theta\rangle$, the existence of real string classes implies the vanishing of the first Pontryagin class
$$
p_1(P) = 0 \in H^4(M,\RR).
$$
Consider now the situation where $M$ is the total space of a principal torus bundle over a base manifold $B$, with fibre $T^k = \RR^k/\ZZ^k$, and $P$ is given by pull-back of a principal $G$-bundle $P_0$ over $B$ 
\begin{equation}\label{eq:PMB}
  \xymatrix{
 P \ar[d]_\pi \ar[r] &  P_0 \ar[d] \\
 M \ar[r]_{p} & B \\
  }
\end{equation}
Then, $P$ has a natural structure of $T^k \times G$-principal bundle, and we can consider $T^k \times G$-invariant string classes on $P$.

\begin{definition}[\cite{BarHek}]\label{def:Tdual}
Let $(M,P,\tau)$ be a triple where $(M,P)$ are as in \eqref{eq:PMB} and $\tau$ is a $T^k \times G$-invariant string classes on $P$. We say that $(M,P,\tau)$ is \emph{T-dual} to another triple $(M',P',\tau')$ if there exists a commutative diagram
\begin{equation*}
  \xymatrix{
 & \ar[ld]^{q} P \times_{P_0} P' \ar[d] \ar[rd]_{q'} & \\
 P \ar[d]_\pi \ar[r] &  P_0 \ar[d] & \ar[l] \ar[d]^{\pi'} P' \\
 M \ar[r]_p & B & M' \ar[l]^{p'} \\
  }
\end{equation*}
and representatives $\hat H$ and $\hat H'$ of the form \eqref{eq:hatH} of the string classes $\tau$ and $\tau'$, respectively, such that
\begin{equation}\label{eq:dF}
dF = q^* \hat H - q'^*\hat H',
\end{equation}
where $F \in \Omega^2(P \times_{P_0} P')$ is a $T^k \times T^{k'}$-invariant two-form on $P \times_{P_0} P'$ inducing a non-degenerate pairing
$$
F \colon \operatorname{Ker} dq \otimes \operatorname{Ker} dq' \to \RR.
$$
\end{definition}

\begin{remark}
Given $(M,P,\tau)$ as in Definition \ref{def:Tdual}, the existence of a T-dual triple implies suitable integrality for the class $\tau$ \cite{BarHek}. For this, one assumes that $\langle \cdot , \cdot \rangle$ is normalized, so that the cohomology class of $\sigma^3$ lifts to an integral cohomology class
$$
[\sigma^3] \in H^3(G,\ZZ).
$$
The integrality of $\tau$ boils down to the existence of a lift to the integral cohomology $H^3(P,\ZZ)$ which restricts to $[\sigma^3] \in H^3(G,\ZZ)$ on the fibres of $P$. This corresponds to the \emph{flux quantization} condition which appears in the string theory literature (cf. \cite{BarHek,GoPro,Witten00}). For simplicity, in this work we will ignore this issue as the existence of a T-dual will follow directly from our ansatz.
\end{remark}

We move next to the relation between T-duality, in the sense of Definition \ref{def:Tdual}, and the existence of solutions of the Hull-Strominger system. Let $(M,P,\tau)$ as before. We assume that $M$ is $2n$-dimensional and spin. Consider the following abstract version of the equations \eqref{eq:HS}: a solution of the Hull-Strominger system on $(M,P)$ is given by a tuple $(\psi,\omega,\theta)$, given by a $SU(n)$-structure $(\psi,\omega)$ on $M$ and a connection $\theta$ on $P$, satisfying equations
\begin{equation}\label{eq:Stromingersystem}
\begin{split}
d\Omega & = 0,\\
F_\theta^{0,2} = 0, \quad F_\theta \wedge \omega^2 & = 0,\\
d(\|\Omega\|_\omega \omega^{n-1}) & = 0,\\
dd^c \omega - \langle F_\theta \wedge F_\theta \rangle & = 0,
\end{split}
\end{equation}
where $\Omega = \|\Omega\|_\omega \psi$. We say that a solution $(\psi,\omega,\theta)$ has string class $\tau$ if
$$
\tau = [- \pi^*d^c \omega + CS(\theta)] \in H^3_{str}(P,\RR).
$$
\begin{theorem}[\cite{GF3}]\label{th:TdualityHS}
Assume that $(M,P,\tau)$ admits a solution of the system \eqref{eq:Stromingersystem} with string class $\tau$. Then, if $(M',P',\tau')$ is a T-dual triple, it also admits a solution with string class $\tau'$.
\end{theorem}

The proof of the previous theorem is based on a characterization of the system \eqref{eq:Stromingersystem} in terms of spinorial equations on a Courant algebroid \cite{GF3}, combined with the salient implications of T-duality in generalized geometry \cite{CaGu}. It gives an explicit formula for the T-dual solution using the two-form $F$ in Definition \ref{def:Tdual}, that we will use in the proof of our main result Theorem \ref{th:TdualtityHSGP}. 

We are now ready to understand the behaviour of the solutions constructed in Theorem \ref{thm1} under T-duality. We need first to introduce some notation, in order to make explicit the dependence on the different parameters. Let $S$ be a K3 surface endowed with a K\"ahler Ricci-flat metric $\omega_S$. Let $\omega_1$ and $\omega_2$ be anti-self-dual $(1,1)$-forms on $S$ such that 
$$
\kappa_j : = [\omega_j/2\pi] \in H^2(S,\ZZ),
$$
for $j = 1,2$. We denote $\kappa = (\kappa_1,\kappa_2)$. Consider the associated principal $T^2$-bundle 
$$
p_{\kappa} \colon X_{\kappa} \to S.
$$
Let $\alpha,t \in \RR$ with $t > 0$ satisfying \eqref{eq:integral}, and let $\mathbb{W}$ be a smooth complex hermitian vector bundle as in Theorem \ref{thm1}. Note that $\mathbb{W}$ depends on the parameters $\alpha,t,\kappa$ only through its second Chern class \eqref{eq:c2Wbis}.
As we will see, this quantity jointly with $\alpha$ will be fixed under T-duality.

Throughout this section, we fix a reference hermitian metric on $\mathbb{V}$. Let $P_0$ denote the bundle of split hermitian frames of 
$$
\mathbb{V} \oplus \mathbb{W}
$$
with trivial determinant. Then, $P_0$ is a $G$-bundle with
$$
G = \operatorname{SU}(3) \times \operatorname{SU}(r),
$$
and we consider a fixed bilinear form on its Lie algebra given by 
$$
\langle \cdot , \cdot \rangle = \frac{\alpha}{4}\tr_{\mathfrak{su}(3)} - \frac{\alpha}{4}\tr_{\mathfrak{su}(r)}.
$$
We define the $G$-bundle $P_\kappa = p_\kappa^* P_0$, which corresponds to the bundle of split hermitian frames of 
$$
T^{1,0}X_{\kappa} \oplus p_\kappa^*\mathbb{W}
$$
with trivial determinant, obtaining a diagram of the form \eqref{eq:PMB}, with $B = S$ and $M = X_\kappa$. 

Consider a solution $(\omega_{t,u},p_\kappa^*\nabla,p^*_\kappa A)$ of the Hull-Strominger system \eqref{eq:HS} given by Theorem \ref{thm1}. We will assume that $\nabla$ is unitary with respect to the fixed hermitian metric on $\mathbb{V}$. Taking the product connection $\theta = p_\kappa^*(\nabla \times A)$, we obtain that $(\Omega,\omega_{t,u},\theta)$ is a solution of \eqref{eq:Stromingersystem} on $(X_{\kappa},P_{\kappa})$ with string class
\begin{equation}\label{eq:strclasssol}
\tau_{\kappa,t} = [- \pi^*d^c \omega_{t,u} + CS(\theta)] \in H^3_{str}(P_\kappa,\RR).
\end{equation}
Note that, in order to obtain a solution of \eqref{eq:Stromingersystem} we crucially use the fact that $\nabla$ is a Hermite-Yang-Mills connection with respect to $\omega_{t,u}$.

\begin{theorem}\label{th:TdualtityHSGP}
Assume that $t\kappa \in H^2(S,\ZZ)^{\oplus 2}$. Then, the solution $(\omega_{t,u},p_\kappa^*\nabla,p^*_\kappa A)$ of the Hull-Strominger system \eqref{eq:HS} on $(X_{\kappa},P_{\kappa})$ with string class $\tau_{\kappa,t}$ is T-dual to a solution on $(X_{\kappa'},P_{\kappa'})$ of the form $(\omega_{t',u},p_{\kappa'}^*\nabla,p^*_{\kappa'} A)$ with string class $\tau_{\kappa',t'}$ defined as in \eqref{eq:strclasssol}, where
\begin{equation}\label{eq:kappat}
\kappa' = - t\kappa, \qquad t' = t^{-1}.
\end{equation}
\end{theorem}

\begin{proof}
Identify $\operatorname{Lie} T^2 = \mathbb{R}^2$ and denote by $\upsilon'$ the connection on $X_{\kappa'}$ with curvature $(-it\omega_1,-it\omega_2)$. Consider the associated complex one-form $\sigma'$ as in \eqref{eq:sigma}. Since
$$
(\kappa,t) \to (- t\kappa,t^{-1})
$$
leaves the second Chern class \eqref{eq:c2Wbis} of $\mathbb{W}$ invariant, we can consider a solution on $(X_{\kappa'},P_{\kappa'})$ of the form $(\omega_{t',u'},p_{\kappa'}^*\nabla,p^*_{\kappa'} A)$, where
$$
\omega_{t',u'} = p_{\kappa'}^*(e^{u'} \omega_S) + \frac{i t'}{2} \sigma' \wedge \overline{\sigma}',
$$
and $u'$ is a solution of the corresponding Laplace equation \eqref{eq:Laplace}. A moment of thought using \eqref{eq:kappat} shows that we can take $u' = u$.

We claim that $(X_{\kappa},P_{\kappa},\tau_{\kappa,t})$ is T-dual to  $(X_{\kappa'},P_{\kappa'},\tau_{\kappa',t'})$, where $\tau_{\kappa',t'}$ is defined as in \eqref{eq:strclasssol}. To see this, we note that
\begin{align*}
d^c \omega_{t,u} & = d^c u \wedge e^u \omega_S - \frac{t}{2}(\omega_1 + i \omega_2) \wedge \overline{\sigma} - \frac{t}{2} \sigma \wedge \overline{(\omega_1 + i \omega_2)}\\
& = d^c u \wedge e^u \omega_S - \frac{t}{2} ((\omega_1,\omega_2) \wedge \upsilon),
\end{align*}
where $(\cdot,\cdot)$ denotes the standard pairing on $\mathbb{R}^2$. Define the three-form
\begin{equation}\label{eq:hatHbis}
\hat H = - \pi^*d^c \omega_{t,u} + CS(\theta),
\end{equation}
and similarly for $\hat H'$. Then, we have
\begin{align*}
q^*\hat H - q'^*\hat H' & = \frac{t}{2} ((\omega_1,\omega_2) \wedge \upsilon) + \frac{1}{2} ((\omega_1,\omega_2) \wedge  \upsilon')\\
& = \frac{1}{2}d (\upsilon \wedge \upsilon'),
\end{align*}
The claim follows by taking $F = \frac{1}{2}(\upsilon \wedge \upsilon')$ in Definition \ref{def:Tdual}. For simplicity in the notation, we have omitted pull-backs in the previous formulae.

To finish the proof, we need to check that the solution $(\omega_{t,u},p_\kappa^*\nabla,p^*_\kappa A)$ is exchanged with the solution $(\omega_{t',u},p_{\kappa'}^*\nabla,p^*_{\kappa'} A)$ via the \emph{T-duality isomorphism} in \cite[Proposition 4.13]{BarHek}. Rather than using the isomorphism between the complicated transitive Courant algebroids determined by the T-dual string classes \cite{BarHek,GF}, we will use the lore of T-\emph{duality commutes with reduction}, and apply the original and more amenable isomorphism between exact Courant algebroids in \cite{CaGu}. Consider the exact $T^2 \times G$-equivariant Courant algebroid
$$
E_\kappa = TP_\kappa \oplus T^*P_\kappa
$$ 
on $P_\kappa$, determined by the three-form \eqref{eq:hatHbis}. Similarly, we define $E_{\kappa'}$ for the corresponding object on $P_{\kappa'}$. Since $(X_{\kappa},P_{\kappa},\tau_{\kappa,t})$ is T-dual to  $(X_{\kappa'},P_{\kappa'},\tau_{\kappa',t'})$, there exists an isomorphism of reduced Courant algebroids \cite[Theorem 3.1]{CaGu}
$$
\varphi  \colon E_\kappa/T^2 \to E_{\kappa'}/T^2
$$
over $P_0$. Explicitly, this is given by
$$
\psi(v + \xi + w\partial_\upsilon + f \upsilon + a + \langle b ,\theta \rangle ) = v + \xi + f\partial_{\upsilon'} + w \upsilon' + a + \langle b ,\theta' \rangle,
$$
where we use the decomposition
$$
E_{\kappa|x} \cong T_sS \oplus T^*_sS \oplus \operatorname{Ker} dp_{\kappa|x} \oplus (\operatorname{Ker} dp_{\kappa|x})^* \oplus \operatorname{Ker} \pi_{x} \oplus (\operatorname{Ker} \pi_{x})^*
$$
at a point $x \in P_\kappa$, over $s \in S$, induced by the connections $\upsilon$ and $\theta$ (and similarly for $E_{\kappa'}$). 

Denote by $g_{t,u}$ the Riemannian metric on $X_\kappa$ determined by the hermitian form $\omega_{t,u}$. Consider the $T^2$-invariant generalized metric on $E_\kappa$ (with indefinite signature)
$$
V_+ = \{ V + \pi^*g_{t,u}(V) + \langle \theta V, \theta \rangle \} \subset E_\kappa
$$
with $V$ running among vector fields on $P_\kappa$. A straightforward calculation shows that (cf. \cite[Example 4.6]{CaGu})
$$
\varphi(V_+) = \{ V + \pi^*g_{t',u}(V) + \langle \theta' V, \theta' \rangle \} \subset E_{\kappa'}.
$$
On the other hand, we note that the reduction of $V_+$ to the transitive Courant algebroid over $X_\kappa$ determined by $\tau_{\kappa,t}$ is a generalized metric inducing the triple $(g_{t,u},d^c\omega_{t,u},\theta)$ (see \cite[Proposition 3.4]{GF}). Thus, using that T-duality commutes with reduction \cite[Proposition 4.13]{BarHek}, we conclude that $(g_{t,u},d^c\omega_{t,u},\theta)$ is exchanged via T-duality with $(g_{t',u},d^c\omega_{t',u},\theta')$. Note that the $SU(3)$-structure of the initial solution is completely determined by the holonomy bundle of the skew-torsion connection
$$
\nabla^{g_{t,u}} - \frac{1}{2} d^c\omega_{t,u}.
$$
This implies that $(\omega_{t',u},p_{\kappa'}^*\nabla,p^*_{\kappa'} A)$ is T-dual to $(\omega_{t,u},p_\kappa^*\nabla,p^*_\kappa A)$.
\end{proof}

\begin{remark}\label{rem:FuYau}
The key to prove Theorem \ref{th:TdualtityHSGP} is that the differential operator in the left hand side of \eqref{eq:Laplace} (given by the Laplacian on $S$) is independent of the size of the fibres $t$, combined with the invariance of the right hand side of \eqref{eq:Laplace} under the transformation \eqref{eq:kappat}. 
With the Chern connection ansatz for $\nabla$, as considered in \cite{FuYau} (see also \cite{Phong4,PPZ2}), the right hand side in Fu-Yau equation \cite[(3.11)]{FuYau} corresponds essentially to the right hand side of \eqref{eq:Laplace}, but the differential operator on the left hand side of \cite[(3.11)]{FuYau} has a non-trivial quadratic dependence on the size of the fibres (cf. Remark \ref{rem:FuYau0}).
\end{remark}

\begin{remark}\label{rem:dilaton}
The \emph{dilaton shift} in \cite[Proposition 6.8]{GF3} plays no role in the proof of Theorem \ref{th:TdualtityHSGP}, since the density induced by the metric solution on the fibres of $X_\kappa$ is constant along the base.
\end{remark}

The behaviour of the parameters \eqref{eq:kappat} under T-duality shows that the size of the fibres for the original solution is exchanged with the topology of the T-dual. Furthermore, the size of the fibres of the T-dual solution is the inverse of the original one. The exchange of big fibres with small fibres is considered the hallmark of a T-duality transformation. The change in the topology of the bundle $\kappa \to -t\kappa$ is reminiscent of the SYZ mirror symmetry 
\cite{SYZ}, which exchanges dual torus fibrations. 
Note that, unlike in SYZ for Calabi-Yau threefolds where dual $T^3$-fibrations are exchanged, here we have dualized with respect  to a $T^2$-fibration on $X_\kappa$. We thank Fei and the referee for these observations. 

Our construction can be easily modified to find T-duals with more general topologies than the ones provided by \eqref{eq:kappat}. For this, we can regard $X_\kappa$ as a U$(1)$-principal bundle over 
$$
Y_\kappa = X_\kappa/\operatorname{U}(1)
$$
and apply the same argument. This way we obtain that for all
\begin{equation}\label{eq:kappatgen}
\kappa' = (t_1\kappa_1, t_2\kappa_2), \qquad t_1,t_2 \in \mathbb{Q}^*, \qquad \kappa' \in H^2(S,\ZZ)^{\oplus 2}
\end{equation}
there is a solution on $(X_{\kappa'},P_{\kappa'})$ related to the one on $(X_{\kappa},P_{\kappa})$ via a chain of T-dualities. The details of this slightly more general construction are left to the reader. Note here that the first Chern classes $\kappa_1,\kappa_2$ determine the fundamental group of $X_\kappa$ through the long homotopy sequence of the fibration, and hence one can easily choose $\kappa'$ such that $X_{\kappa}$ and $X_{\kappa'}$ have different fundamental group. Tipler has pointed out to the author that the relation \eqref{eq:kappatgen} yields an isomorphic space of complex deformations for $X_\kappa$ and $X_{\kappa'}$ and also isomorphic Aeppli cohomology groups $H^{1,1}_A(X_{\kappa'}) \cong H^{1,1}_A(X_{\kappa})$. In the light of the new developments in \cite{grt2}, this provides a first consistency check that our T-dual solutions correspond to a chain of (0,2)-mirrors. We hope to go back to this question elsewhere.

The T-dual solutions we have found in Theorem \ref{th:TdualtityHSGP} are rather special, since the connections $\nabla$ and $A$ in \eqref{eq:HS} are pull-backs of connections on the base of the torus fibration. This has the effect that the bundles act as mere \emph{spectators} in the T-dualization of the solution. In a more general situation, as explored in the physics literature \cite{EsvlinMina,Israel,MaSp} and independently proposed in \cite{BarHek}, one would expect that the Chern classes of the bundle in the original solution contribute to the topology of the T-dual. It would be interesting to find examples of solutions of the Hull-Strominger system which exhibit this phenomenon. In the light of \cite[Proposition 17]{FuYau}, this does not seem to be possible with the Goldstein-Prokushkin ansatz for the hermitian metric.

\end{document}